\documentclass[12pt,twoside]{article}
\usepackage{amsmath,amsfonts,amssymb,amsthm}
\usepackage{bbm}
\usepackage{graphicx}
\usepackage{hyperref}
\usepackage{color}
\newcommand{\R}{\mathbf{R}}

\newcommand{\h}{\overline h}
\newcommand{\hh}{\overline h}

\DeclareMathOperator{\spt}{spt}
\def\loc{{\mathrm{loc}}}

\newcommand{\Lip}{{\mathrm{Lip}_c^{\h}}}

\newcommand{\Liph}{{\mathrm{Lip}_c^{\hh}}}
\newcommand{\ho}{h_{0}}
\newcommand{\he}{h_{\eps}}

\newcommand{\W}{\overline W}
\newcommand{\eps}{\varepsilon}
\newcommand{\la}{\langle}
\newcommand{\ra}{\rangle}
\newcommand{\lla}{\langle\langle}
\newcommand{\rra}{\rangle\rangle}

\newcommand{\ue}{u_{\eps}}
\newcommand{\ve}{v_{\eps}}
\newcommand{\we}{w_{\eps}}
\newcommand{\We}{W_{\eps}}
\newtheorem{prop}{Proposition}
\newtheorem{theorem}{Theorem}
\newtheorem{lemma}{Lemma}

\begin{document}

\title{A penalization approach to linear programming duality with application to
    capacity constrained transport}

\author{Jonathan Korman, Robert J. McCann, and Christian 
Seis\thanks{Department of Mathematics, University of Toronto, Toronto ON Canada M5S 2E4
 \copyright 2013 by the authors. 
The authors are grateful to Terry Rockafellar for insightful comments on an earlier
version of this manuscript. They are also pleased to acknowledge partial support by
NSERC grant 217006-08. RJM acknowledges the kind hospitality of the Mathematical
Sciences Research Institute at Berkeley CA, where part of this work was performed.}}

\maketitle
\abstract{A new approach to linear programming duality is proposed which relies on
quadratic penalization,  so that the relation between solutions to the penalized primal
and dual problems becomes affine.  This yields a new proof of Levin's
duality theorem for capacity-constrained optimal transport as an infinite-dimensional
application.}

\section{Introduction}

Given a distribution of sources (manufacturers) $f(x)$ and sinks (consumers) $g(y)$, and a function $c(x,y)$ that measures the cost of transporting a unit of mass from $x\in \R^m$ to $y\in \R^n$, the optimal transport problem of Monge \cite{Monge81} and Kantorovich \cite{Kantorovich42} seeks to minimize the total cost required to transport $f$ to $g$. We consider a variant of that classical problem by imposing a limitation on the amount of mass that is allowed to be transferred from $x$ to $y$: The {\em capacity constrained} optimal transport problem.

\medskip

For two given probability distributions $f\in L^{1}(\R^m)$ and $g\in L^{1}(\R^n)$, and a fixed nonnegative function $\overline h\in L^{\infty}(\R^m\times \R^n)$, we denote by $\Gamma^{\overline h}(f,g)$ the set of all nonnegative measurable joint densities that are bounded by $\overline h$, i.e., $f(x) = \int h(x,y)\, dy$, $g(y)= \int h(x,y)\, dx$, and $0\le h\le\h$. Necessary and sufficient conditions for $\Gamma^{\overline h}(f,g)$ to be nonempty are given by Kellerer \cite{Kellerer64a,Kellerer64b} and Levin \cite{Levin84}, namely $\Gamma^{\overline h}(f,g)\not=\emptyset$ if and only if
\[
f(A)+g(B)-\bar{h}(A \times B) \leq 1 \quad\mbox{for any Borel measurable }A\subset\R^m,\, B\subset \R^n. 
\]
Throughout this article, we will always assume that these conditions are satisfied. Finally, let $c\in L_{\loc}^1(\R^m\times\R^n)$ and
\[
I_c(h):= \iint c(x,y) h(x,y)\, dxdy.
\]
The {\it optimal transportation problem with capacity constraints} consists in finding and studying optimal transference plans $\ho\in\Gamma^{\h}(f,g)$ for the total cost functional $I_c$:
\[
I_c(\ho)\;=\; \min_{h\in\Gamma^{\h}(f,g)} I_c(h).
\]

\medskip

Optimal transference plans always exist as can be easily established via the direct method of calculus of variations. Regarding the capacity constrained optimal transport as an infinite-dimensional linear programming problem, it is not surprising that minimizers are extreme points of the convex polytope $\Gamma^{\h}(f,g)$. They can be characterized by $\ho=\h\chi_W$ for some Lebesgue measurable set $W$ in $\R^m\times \R^n$ \cite{KormanMcCann13}. Under suitable conditions on the cost function $c$, minimizers are unique \cite{KormanMcCann12}.

\medskip

In this short manuscript, we address the linear programming duality for capacity constrained optimal transport. Although such a duality was already established by Levin\footnote{In a private communication, Rachev and R\"uschendorf attribute Theorem
4.6.14 of \cite{RachevRuschendorf98} to a handwritten manuscript of Levin;
we are unsure where or whether it was subsequently published.}
 (see Theorem 4.6.14 of \cite{RachevRuschendorf98}), we present an alternative proof here. While Rockafellar-Fenchel dualities (including Levin's, and the Kantorovich's duality for classical optimal transport, cf.\ \cite[Ch.\ 1]{Villani1}) are usually proved using an asbtract minimax argument with the Hahn--Banach theorem at its core, our new proof is rather elementary and is based on a quadratic approximation of the linear program, cf.\ Section \ref{S2}. Combining the techniques presented in the following with some of the results derived by the authors in a companion paper \cite{KormanMcCannSeis13b}, we also provide a new {\em elementary} proof of Kantorovich's duality.

\medskip

We prove Levin's duality under the additional assumption that the capacity bound $\h$ is compactly supported, and we write $\W=\spt(\h)$. Notice that under this hypothesis, $h,\, f$, and $g$ are bounded and compactly supported, so that in particular $f,\, g,\ h\in L^p$ for any $1\le p\le \infty$.

\medskip

Before stating Levin's duality theorem, we introduce some notation. Given a function $\zeta=\zeta(x,y)$ defined on $\R^m\times\R^n$, we write $\la\zeta\ra_x$ and $\la\zeta\ra_y$ for the $x$- and $y$-marginals of $\zeta$, i.e., $\la \zeta\ra_x:=\int \zeta(x,y)\, dy$ and 
$\la \zeta\ra_y:=\int \zeta(x,y)\, dx$. The integral over the product space is denoted by $\lla\zeta\rra$, i.e., $\lla \zeta\rra:=\iint \zeta(x,y)\, dxdy$. Likewise, if $\zeta=\zeta(x)$ or $\zeta=\zeta(y)$, we simply write $\la\zeta\ra$ to denote the integral over $\R^m$ or $\R^n$, respectively.
With the above notation, the total cost functional becomes
\[
I_c(h)\;:=\; \lla ch\rra.
\]

\medskip

We introduce some further notation. Let
\[
J(u,v,w)\;:=\;  - \la uf \ra - \la vg\ra + \lla w \h\rra,
\]
and
\begin{eqnarray*}
\Lip&=&\left\{ (u,v,w)\in L^1(f\, dx)\times L^1(g\,dy)\times L^1(\bar h \,dxdy):\right.\\
&&\mbox{}\left.\: u(x) + v(y) - w(x,y) + c(x,y)\ge 0\mbox{ and }w(x,y)\le0\right\}.
\end{eqnarray*}
Here, we use the notation that $L^1(\mu)$ is the class of all absolutely integrable functions with respect to the measure $\mu$. Obviously, $J(u,v,w)$ is well-defined on $\Lip$.

\medskip

Our main result is the following

\begin{theorem}[Levin's duality]\label{T1}Let $0\le \h\in L^{\infty}(\R^m\times \R^n)$ compactly supported and $f\in L^1(\R^m)$ and $g\in L^1(\R^n)$ be two probability densities such that $\Gamma^{\h}(f,g)\not=\emptyset$. Suppose that $c\in L^1_{\loc}(\R^m\times \R^n)$.
Then
\[
\min_{h\in\Gamma^{\h}(f,g)} I_c(h)\;=\; \sup_{(u,v,w)\in \Lip} J(u,v,w).
\]
\end{theorem}

In \cite{KormanMcCannSeis13b}, the authors prove that the supremum on the right is attained by triple of measures of finite total variation which --- so far as we know --- need not generally be absolutely continuous with respect to Lebesgue.

\medskip

In the following Section \ref{S2}, we illustrate the method of this paper by considering an
analogous problem in finite-dimensions. The proof of Theorem \ref{T1} is presented in Section \ref{S3}.

\section{Finite-dimensional linear programming duality}\label{S2}
In this section we illustrate the method of this paper by providing a non-standard proof of the finite-dimensional linear programming duality, which is new, as far as we know. For $A\in \R^{m\times n}$, $c\in \R^n$, and $b\in \R^m$, duality asserts
\begin{equation}\label{19}
\inf_{y\ge 0,\, A^T y=c} b\cdot y\;=\; \sup_{Ax\le b} c\cdot x,
\end{equation}
where, of course, $x\in \R^n$ and $y\in \R^m$, cf. \cite[Ch.\ 4]{Luenberger}. We understand the inequalities $y\ge0$ and $Ax\le b$ componentwise. The advantage of our proof of \eqref{19} is that it 
generalizes in a straightforward way to certain infinite-dimensional problems, as we will see in the subsequent section. There we give a new proof of Levin's duality, Theorem \ref{T1}.

\medskip

The basic idea in our proof of \eqref{19} is to relax the equality constraint in the minimization problem by adding a penalizing quadratic term to the linear function. That is, we consider the quadratic function
\[
I^{\eps}(y):= b\cdot y + \frac1{2\eps}|A^T y-c|^2,
\]
and minimize $I^{\eps}$ over all $y$ such that $y\ge0$. Relaxing a minimization problem by approximating hard by soft constraints is a fairly standard procedure in the calculus of variations whether it be to regularize singular problems (e.g.\ \cite{BBH}) or simply to extend the class of admissible competitors (e.g.\ \cite{MSS}). In particular, when dealing with constraints of different kinds, as in the capacity constrained optimal transport problem, relaxing some of these constraints eventually simplifies the computation of the Euler--Lagrange equation dramatically, see e.g.\ Lemma \ref{L3} below.

\medskip

The key observation in our analysis is a duality theorem for the relaxed problem,
\begin{equation}\label{20}
\min_{y\ge0} I^{\eps}(y)\;=\;\max_{Ax\le b} J^{\eps}(x),
\end{equation}
provided that the minimum on the left is attained, and where $J^{\eps}(x) = c\cdot x - \frac{\eps}2 |x|^2$. The derivation of the ``$\inf\ge\sup$''-inequality is standard: Using $y\ge 0$ and $b\ge Ax$, we have
\begin{eqnarray*}
I^{\eps}(y) & \ge& Ax\cdot y + \frac1{2\eps}|A^T y-c|^2\\
&=& c\cdot x + x\cdot\left(A^T y-c\right) + \frac1{2\eps}|A^T y-c|^2\\
&=& c\cdot x -\frac{\eps}2 |x|^2  + \frac{\eps}2|x + \frac1{\eps}\left(A^Ty-c\right)|^2\\
&\ge& J^{\eps}(x),
\end{eqnarray*}
and the statement follows upon taking the infimum on the left and the supremum on the right. Moreover, the above inequality turns into an equality for any pair $(y_{\eps},x_{\eps})$ with $y_{\eps}\ge0$, $Ax_{\eps}\le b$ and
\begin{equation}\label{21b}
x_{\eps}=\frac1{\eps} \left(c - A^T y_{\eps}\right).
\end{equation}
In particular, if such a pair exists, we must have
\[
\min_{y\ge0} I^{\eps}(y)\;=\; I^{\eps}(y_{\eps})\;=\; J^{\eps}(x_{\eps})\;=\; \max_{Ax\le b} J^{\eps}(x),
\]
that is \eqref{20} holds. The existence of $(y_{\eps},x_{\eps})$ with $y_{\eps}\ge0$, $Ax_{\eps}\le b$, and \eqref{21b} is a simple but crucial insight, and follows from a direct derivation of the first-order necessary condition for existence of minimizers of $I^{\eps}$. Whether the minimum is attained in this finite-dimensional toy problem certainly depends on the particular choice of the matrix $A$. Notice, however, that existence of minimizers is obvious when including the ``capacity constraint'' $y\le\bar y$ for some $\bar y\in \R^m$ into the problem, which would actually correspond to the real finite-dimensional analog for the problem considered in this paper. To keep the discussion in this section as elementary as possible, we simply drop this capacity constraint and assume the existence of a minimizer $y_{\eps}$ of $I^{\eps}$ in the following.

\medskip

If $y_{\eps}$ is a minimizer of $I^{\eps}$ under the constraint $y\ge0$, and $\xi$ is an arbitrary vector in $\R^m$ such that $\xi^i\ge 0$ if $y_{\eps}^i=0$, we have $y_{\eps} + s\xi\ge 0$ for all $s>0$ sufficiently small, and thus $I^{\eps}(y_{\eps})\le I^{\eps}(y_{\eps} + s\xi)$. Consequently,
\[
0\;\le\; \left.\frac{d}{ds}\right|_{s=0} I^{\eps}(y_{\eps} + s\xi)\;=\; \left(b + \frac1{\eps} A\left(A^Ty_{\eps}-c\right)\right)\cdot \xi.
\]
By the choice of $\xi$, this implies that
\[
b + \frac1{\eps} A\left(A^Ty_{\eps}-c\right)\;\ge\;0,
\]
and thus, $x_{\eps}$ defined as in \eqref{21b} satisfies $b\ge Ax_{\eps}$. Hence, $(y_{\eps},x_{\eps})$ is dual pair with the desired properties. This proves \eqref{20} under the assumption that the minimum of  $I^{\eps}$ is attained.

\medskip

There is a remarkable {\em affine} relation \eqref{21b} between the maximizer $x_{\eps}$ of $J^{\eps}$ and the minimizer $y_{\eps}$ of $I^{\eps}$. This relation, however, is not surprising, since \eqref{21b} can be also be derived as the first order necessary condition for the dual maximum problem, which is linear in $x_{\eps}$ since $J^{\eps}$ is quadratic, and in which $y_{\eps}$ plays the role of the Lagrange multiplier associated with the constraint $Ax\le b$.

\medskip

It remains to pass to the limit $\eps\downarrow 0$ in \eqref{20} to obtain \eqref{19}, but we omit the details at this point.

\section{Proof of Levin's duality theorem}\label{S3}

Theorem \ref{T1} is an immediate consequence of the following two Propositions:
\begin{prop}\label{P1}The hypotheses of Theorem \ref{T1} imply
\begin{equation}\label{1}
\inf_{h\in\Gamma^{\h}(f,g)} I_c(h)\;\ge\; \sup_ {(u,v,w)\in \Lip} J(u,v,w).
\end{equation}
\end{prop}

\begin{prop}\label{P2}The hypotheses of Theorem \ref{T1} imply existence of a sequence $\{(\ue ,\ve,\we )\}_{\eps\downarrow0}$ in $\Lip$ such that
\begin{equation}\label{9}
I_c(\ho)\;=\;  \lim_{\eps\downarrow 0}J(\ue,\ve,\we),
\end{equation}
where $\ho$ is a minimizer of the form $\ho=\h\chi_W$.
\end{prop}

The first Proposition is easily established:

\begin{proof}[Proof of Proposition \ref{P1}]
For any coupling $h\in \Gamma^{\h}(f,g)$  with $I_c(h)$ finite, and $(u,v,w)\in\Lip$ we have
\begin{eqnarray*}
I_c(h) &=& - \la uf\ra - \la vg\ra + \lla w\h\rra+\lla \left(c +u+v-w\right)h\rra + \lla w(h-\h)\rra\\
&\ge& J(u,v,w),
\end{eqnarray*}
where in the first line we have used the marginal constraint on $h$ and in the second line we applied the definition of $\Lip$ together with the fact that $0\le h\le\h$. Now, the inequality in \eqref{1} follows immediately upon taking the supremum on the right and the infimum on the left.
\end{proof}

The remainder of the paper is devoted to the proof of Proposition \ref{P2}. 


\medskip

We introduce a {\em relaxed} version of the optimal transportation problem with capacity constraints. Let $\eps>0$ denote a small number. We define the relaxed transportation cost
\begin{eqnarray*}
I_c^{\eps}(h)&=&\lla ch\rra + \frac1{2\eps}\| \la h\ra_x-f\|_2^2+\frac1{2\eps}\ \|\la h\ra_y-g \|_2^2
\end{eqnarray*}
using the $L^2$ norms $\| \cdot \|_2$ on $\R^m$ and $\R^n$.
Notice that $I_c^{\eps}(\ho)= I_c(\ho)$. Furthermore, for $(u,v,w)\in\Liph$ such that $u$ and $v$ are both square-integrable, we consider the functional
\[
J^{\eps}(u,v,w)\;:=\; -\la uf\ra - \la vg\ra + \lla w\hh\rra - \frac{\eps}2 \| u\|_2^2 - \frac{\eps}2 \| v\|_2^2.
\]
We can extend $J^{\eps}$ to a functional all over $\Liph$ by setting $J^{\eps}(u,v,w):=-\infty$ if $u\not\in L^2(\R^m)$ or $v\not\in L^2(\R^n)$.

\medskip

In a first step, we derive the analogous statement to Proposition \ref{P1} for the relaxed problem.

\begin{lemma}[Easy direction of relaxed duality]
\label{L1}
For $\epsilon>0$, the hypotheses of Theorem \ref{T1} imply
\begin{equation}\label{2}
\inf_{0\le h\le \hh} I^{\eps}_c(h)\;\ge\; \sup_ {(u,v,w)\in \Liph} J^{\eps}(u,v,w).
\end{equation}
\end{lemma}

\begin{proof}Without lost of generality we may choose $0\le h\le\hh$ and $(u,v,w)\in \Liph$ such that $I_c^{\eps}(h)$ and $J^{\eps}(u,v,w)$ are both finite. A short computation shows that $I_c^{\eps}(h)$ can be rewritten as
\begin{eqnarray*}
I_c^{\eps}(h) &=& -\la uf\ra- \la vg\ra + \lla w\hh\rra - \frac{\eps}2\|u\|_2^2 - \frac{\eps}2\|v\|_2^2\\
&&\mbox{} + \lla \left(c+u+v-w\right)h\rra + \lla w(h-\hh)\rra\\
&&\mbox{} + \frac1{2\eps}\left\| \la h\ra_x-f-\eps u\right\|_2^2+ \frac1{2\eps} \left\|\la h\ra_y-g-\eps v\right\|_2^2.
\end{eqnarray*}
By the definition of $\Liph$ and $J^{\eps}(u,v,w)$, recalling that $0\le h\le\hh$, and observing that the term in the last line is trivially nonnegative, it follows that
\[
I_c^{\eps}(h)\;\ge\; J^{\eps}(u,v,w).
\]
Taking the infimum on the left hand side and the supremum on the right hand side yields \eqref{2}.

\end{proof}

\medskip

We next address existence of minimizers for the relaxed problem.

\begin{lemma}[Existence of minimizers and uniqueness of relaxed marginals]
\label{L2}
The hypotheses of Lemma \ref{L1} imply 
existence of a minimizer $\he$ of $I_c^{\eps}$, and $\he$ can be chosen of the form $\he=\h\chi_{W_{\eps}}$ for some Lebesgue measurable set $\We$ in $\R^m\times\R^n$. Moreover, if $\tilde h_{\eps}$ is another minimizer of $I_c^{\eps}$, then $\la\he\ra_x = \la\tilde h_{\eps}\ra_x$ and $\la\he\ra_y=\la\tilde h_{\eps}\ra_y$.
\end{lemma}

Existence of minimizers and uniqueness of their marginals follow by standard arguments. We provide the proof for the convenience of the reader.

\begin{proof} Since $\ho$ is admissible for $I_c^{\eps}$ with $I_c^{\eps}(\ho) = I_c(\ho)$, it follows that $-\|\h\|_{\infty} \|c\|_{L^1(\W)}\le\inf I^{\eps}_c(h)\le I_c(\ho)<\infty$, where the infimum is taken over all admissible $h$. Let $\{h_{\nu}\}_{\nu\uparrow\infty}$ denote a minimizing sequence. By the $L^{\infty}$-constraint $0\le h_{\nu}\le\h$, we can find an $L^{\infty}$-function $ h_{\eps}$ satisfying $0\le  h_{\eps}\le \h$ and we can extract a subsequence converging to $ h_{\eps}$ weakly-$\star$ in $L^{\infty}$. Moreover, as the sequences $\{\la h_{\nu}\ra_x-f\}_{\nu\uparrow\infty}$ and $\{\la h_{\nu}\ra_y-g\}_{\nu\uparrow\infty}$ are both bounded in $L^2$, we may extract a further subsequence ensuring that these sequences converge weakly in $L^2$ towards $\la h_{\eps}\ra_x-f$ and $\la h_{\eps}\ra_y-g$, respectively. Without relabeling the subsequences, we then have
\begin{eqnarray*}
\|\la h_{\eps}\ra_x-f\|_2&\le& \liminf_{\nu\uparrow\infty} \|\la h_{\nu}\ra_x-f\|_2,\\
\|\la h_{\eps}\ra_y -g\|_2&\le& \liminf_{\nu\uparrow\infty} \|\la h_{\nu}\ra_y -g\|_2,
\end{eqnarray*}
by the lower semi-continuity of the $L^2$ norm with respect to weak $L^2$ convergence. Moreover, since $c\in L^1_{\loc}$ and $h_{\eps}, h_{\nu}$ are supported in $\W$, weak-$\star$ convergence guarantees that
\[
\lla c h_{\eps} \rra = \lim_{\nu\uparrow\infty} \lla c h_{\nu} \rra.
\]
Hence, by combining the above (in)equalities, we have
\[
I_c^{\eps}(h_{\eps})\;\le\; \liminf_{\nu\uparrow\infty} I_c^{\eps}(h_{\nu}).
\]
Since $\{h_{\nu}\}_{\nu\uparrow\infty}$ was a minimizing sequence, it turns out that $h_{\eps}$ minimizes $I_c^{\eps}$. By strict convexity of the relaxed optimization problem, $h_{\eps}$ has unique marginals. Moreover, since $\he$ minimizes $I_c$ in the class $\Gamma^{\h}(\la \he\ra_x, \la \he\ra_y)$, we can choose $\he$ geometrically extreme (with respect to $\ho$): $\he=\h\chi_{\We}$ for some Lebesgue measurable set $\We\subset\W$, cf.\ \cite{KormanMcCann13}.
\end{proof}

In the following, we construct an approximate dual triple $(u_{\eps}, v_{\eps}, w_{\eps})$ by defining
\begin{eqnarray}
u_{\eps}&:=& \frac1{\eps}\left(\la\he\ra_x-f\right),\label{3}\\
v_{\eps} &:=& \frac1{\eps}\left( \la\he\ra_y-g\right),\label{4}\\
w_{\eps}&:=&\min\{c+ \ue+\ve,0\}\label{5}.
\end{eqnarray}
The definition of $\we$ entails that $c +u_{\eps} + v_{\eps} -w_{\eps} \ge0$ and $\we\le0$. Observe that by Lemma \ref{L2} these triples are determined independently of the choice of $\he$. 
Notice that $u_\eps$ and $v_\eps$ (but note $w_\eps$) depend linearly on $h_\eps$, echoing our finite
dimensional model problem.
In  Lemma \ref{L6} below, we prove that this triple maximizes $J^{\eps}$ in $\Liph$, which in turn yields the duality theorem for the relaxed problem. We can pass to the limit $\eps\downarrow0$ in this duality to prove Proposition~\ref{P2}.

\begin{lemma}[Euler--Lagrange equations for relaxed problem]\label{L3}
Taking $\he$ and $W_{\eps}$ from Lemma~\ref{L2}, using
\eqref{3}--\eqref{5} to define  $(u_\eps,v_\eps,w_\eps)$ yields
\begin{equation}
\label{6}
c + u_{\eps} + v_{\eps}\;\left\{\begin{array}{ll}\le\; 0&\quad\mbox{a.e.\ in } W_{\eps},\\
\ge\; 0&\quad\mbox{a.e.\ in }\W\setminus W_{\eps}.\end{array}\right.
\end{equation}
\end{lemma}

\begin{proof}[Proof of Lemma \ref{L3}]
Let $\zeta \ge0$ denote an arbitrary smooth test function. We give the argument for the second inequality in \eqref{6} by considering the outer perturbation
\[
h_{\eps}^{\sigma}:= \he + \sigma \zeta(\h-\he)=\left\{ \begin{array}{ll} \he &\quad\mbox{a.e.\ in }W_{\eps},\\
\sigma \zeta\h &\quad\mbox{a.e.\ in }\W\setminus W_{\eps}.\end{array}\right.
\]
Obviously $\he^0=\he$ and $0\le h_{\eps}^{\sigma}\le \h$ for $0\le \sigma\le \|\zeta\|_{\infty}^{-1}$. Hence,  by the optimality of $\he$ we have $I^{\eps}_c(\he^0)\le I^{\eps}_c(h_{\eps}^{\sigma})$, and a short computation using \eqref{3}\&\eqref{4} yields
\[
0\;\le\; \left.\frac{d}{d\sigma}\right|_{\sigma=0} I^{\eps}_c(h_{\eps}^{\sigma})\; =\;
\lla \left(c + u_{\eps} + v_{\eps}\right)\zeta(\h-\he)\rra.
\]
This estimate holds for all smooth test functions $\zeta\ge0$. Via the Fundamental Lemma of Calculus of Variations it immediately follows  that $\left(c + u_{\eps} + v_{\eps}\right)(\h-\he)\ge0$ almost everywhere. Moreover, since $\h-\he$ is nonnegative almost everywhere and positive almost everywhere in $\W\setminus \We$, we deduce the second inequality in \eqref{6}.

\medskip

The argument for the first inequality in \eqref{6} is proved similarly, we just need to consider the perturbation $h_{\eps}^{\sigma}:= \he - \sigma\zeta\he$ and argue as above.
\end{proof}


\begin{lemma}[A duality theorem for the relaxed problem]\label{L6}
Taking $\he$ and $W_{\eps}$ from Lemma~\ref{L2} and using
\eqref{3}--\eqref{5} to define  $(u_\eps,v_\eps,w_\eps)$ yields
\begin{equation}\label{7}
I_c^{\eps}(\he)\;=\; J^{\eps}(\ue,\ve,\we).
\end{equation}
In particular, $(\ue,\ve,\we)$ maximizes $J^{\eps}(u,v,w)$ in $\Liph$.
\end{lemma}

\begin{proof}
Using the definition of $\ue$ and $\ve$, we easily compute that
\[
J^{\eps}(\ue,\ve,\we)
\;=\; I_c^{\eps}(\he)  -\lla\left(c + \ue + \ve- \we\right)\he\rra + \lla \we (\hh - \he)\rra.
\]
In view of \eqref{5} and \eqref{6} we see that $(c + \ue + \ve - \we)\he\equiv0$ and $ \we (\hh - \he)\equiv 0$. Hence, \eqref{7} follows.

\medskip

In view of \eqref{2}, the triple $(\ue,\ve,\we)$ is a maximizer of $J^{\eps}$ in $\Liph$ because $(\ue,\ve,\we)\in\Liph$ by construction.
\end{proof}


The next result shows that solutions to the relaxed problem approximate the original one ``as the soft constraints become harder''.

\begin{lemma}[Extracting a limit from the penalized problems]\label{L8}
The sequence $\{\he\}_{\eps\downarrow0}$ defined by Lemma \ref{L2}
is precompact in the $L^{\infty}$-weak-$\star$ topology and every limit point $\ho$ is a minimizer of $I_c$. Moreover,
\begin{eqnarray}
\lim_{\eps\downarrow 0} I_c(\he)&=&I_c(\ho),\label{16}\\
\lim_{\eps\downarrow 0} \eps \|\ue\|_2^2&=&0,\label{17}\\
\lim_{\eps\downarrow 0} \eps \| \ve\|_2^2&=&0.\label{18}
\end{eqnarray}
\end{lemma}

\begin{proof}
Since $0\le \he\le \hh$, we immediately see that a subsequence of $\{\he\}_{\eps\downarrow0}$ (which we will not relabel) converges weakly-$\star$ in $L^{\infty}$ to some function $0\le\tilde h\le\h$. 

\medskip

By the optimality of $\he$ and since any minimizer $\tilde h_0$ of the original $\eps=0$ problem is admissible in the relaxed problem, we have the trivial bound
\begin{equation}\label{10}
I_c^{\eps}(\he)\;\le \; I_c^{\eps}(\tilde h_0)\;=\; I_c(\tilde h_0),
\end{equation}
and thus weak-$\star$ convergence of $\{\he\}_{\eps>0}$ implies that
\[
\lla c\tilde h\rra\; = \; \lim_{\eps\downarrow 0} \lla c\he\rra\; \stackrel{\eqref{10}}{ \le}\;\lla c \tilde h_0\rra,
\]
i.e., $I_c(\tilde h)\le I_c(\tilde h_0)$. Since $0\le \tilde h\le \h$, it remains to show that $\tilde h$ satisfies the marginal constraints
\[
f = \la \tilde h\ra_x\quad\mbox{and}\quad g=\la \tilde h\ra_y,
\]
because then $\tilde h$ must be a minimizer of $I_c$, i.e., $I_c(\tilde h)= I_c(\tilde h_0)$.

\medskip

Indeed, from \eqref{10} we deduce that
\[
\|f - \la\he\ra_x\|_2^2+ \|g-\la\he\ra_y\|_2^2\;\le \; 2\eps \lla c \tilde h_0\rra,
\]
which states that $\la\he\ra_x\to f$ and $\la\he\ra_y\to g$ in $L^2$. For any smooth and compactly supported testfunction $\zeta=\zeta(x)$, we write
\[
\la (f-\la \tilde h\ra_x)\zeta\ra
\;=\;\la (f-\la\he\ra_x)\zeta\ra+\la(\la\he\ra_x - \la \tilde h\ra_x)\zeta\ra.
\]
The first integral on the right converges to zero by the $L^2$-convergence of the marginals stated above. The second integral can be rewritten as
$
\lla (\he-\tilde h)\zeta\rra
$
which converges to zero by $L^{\infty}$-weak-$\star$ convergence. Invoking the Fundamental Lemma of Calculus of Variations, this proves that $f=\la \tilde h\ra_x$, and the analogous argument applies for the $y$-marginals, showing that $g=\la \tilde h\ra_y$.

\medskip

Since $I_c(\tilde h)\le\liminf_{\eps\downarrow 0} I_c^{\eps}(\he)$, passing to the limit in \eqref{10}, the above analysis shows that
\[
\min_{h\in\Gamma^{\h}(f,g)} I_c(h)\;=\; \lim_{\eps\downarrow0}I_c(\he)\;=\;\lim_{\eps\downarrow0}I^{\eps}_c(\he),
\]
which implies \eqref{16}--\eqref{18} by the definition of $\ue$ and $\ve$.
\end{proof}


We are now in the position to proof Proposition \ref{P2}.

\begin{proof}[Proof of Proposition \ref{P2}] We may rewrite identity \eqref{7} in terms of $J(\ue,\ve,\we)$ and $I_c(\he)$, that is
\[
J(\ue,\ve,\we)\;=\; I_c(\he) + \eps\|\ue\|_2^2 + \eps \| \ve\|_2^2.
\]
Invoking \eqref{16}--\eqref{18}, we then have
\[
\lim_{\eps\downarrow0}J(\ue,\ve,\we)\;=\; I_c(\ho),
\]
i.e., equation \eqref{9}. It remains to recall that $(\ue,\ve,\we)\in\Lip$ by Lemma~\ref{L6}.
\end{proof}

\bibliographystyle{acm}
\bibliography{OT}

\end{document}